\documentclass[11pt]{amsart}

\usepackage{macros}

\usepackage[alphabetic,initials]{amsrefs}

\makeatletter
\let\oldabs\abs
\def\abs{\@ifstar{\oldabs}{\oldabs*}}
\makeatother

\newcommand{\Var}{\mathbf{Var}}

\begin{document}

\title{Nowhere vanishing $1$-forms on varieties admitting a good minimal model}
\author{Benjamin Church}
\address{Department of Mathematics, Stanford University, 
380 Jane Stanford Way, CA 94086} 
\email{{\tt bvchurch@stanford.edu}}

\thanks{}
\date{\today}


\begin{abstract}
We prove several conjectures relating the existence of nonvanishing $1$-forms to smooth morphisms over abelian varieties, assuming the existence of good minimal models. The proof involves a decomposition result for a family of Calabi-Yau varieties equipped with a surjective map to an abelian scheme. In the uniruled case, supposing the MRC base admits a good minimal model, we also achieve a structure theorem for those varieties admitting nowhere vanishing $1$-forms.
\end{abstract}

\maketitle

\makeatletter
\newcommand\@dotsep{4.5}
\def\@tocline#1#2#3#4#5#6#7{\relax
  \ifnum #1>\c@tocdepth 
  \else
    \par \addpenalty\@secpenalty\addvspace{#2}%
    \begingroup \hyphenpenalty\@M
    \@ifempty{#4}{%
      \@tempdima\csname r@tocindent\number#1\endcsname\relax
    }{%
      \@tempdima#4\relax
    }%
    \parindent\z@ \leftskip#3\relax
    \advance\leftskip\@tempdima\relax
    \rightskip\@pnumwidth plus1em \parfillskip-\@pnumwidth
    #5\leavevmode\hskip-\@tempdima #6\relax
    \leaders\hbox{$\m@th
      \mkern \@dotsep mu\hbox{.}\mkern \@dotsep mu$}\hfill
    \hbox to\@pnumwidth{\@tocpagenum{#7}}\par
    \nobreak
    \endgroup
  \fi}
\def\l@section{\@tocline{1}{0pt}{1pc}{}{\bfseries}}
\def\l@subsection{\@tocline{2}{0pt}{25pt}{5pc}{}}
\makeatother


\tableofcontents

\section{Introduction}

A theorem of Popa and Schnell \cite{PS14} answering the conjectures of Hacon and \Kovacs \cite{HK05} and Luo and Zhang \cite{LZ05} shows that if a smooth projective variety admits a $1$-form without zeros it cannot be of general type. However, we expect far more stringent constraints on those varieties of intermediate Kodaira dimension that actually do carry $1$-forms without zeros. For example, Hao and Schreieder exactly classified all $3$-folds with this property \cite{HS21(1)} and found they are always blowups of isotrivial smooth fibrations over an abelian variety. 
\par
Inspired by the case of $3$-folds, our guiding principle is that $1$-forms without zeros should be explained geometrically by the existence of a smooth map to an abelian variety. In the non-minimal case, this is not true on the nose (see Example~\ref{example:Schreieder_Yang}) and one must first pass to a birational model in order to obtain a \textit{smooth} map to an abelian variety. This principle was explored in \cite{Hao23} and \cite{CCH23} where results were obtained for large Kodaira dimension. We generalize those results (see \S\ref{section:iitaka}) and remove the assumption on the Kodaira dimension. We apply such results to prove -- assuming the main conjectures of MMP -- several conjectures in the spirit of our main principle. In particular, this resolves the conjectures unconditionally in dimension at most $4$ using the results of Fujino \cite{Fuj10} on irregular $4$-folds.
\par 
First, we prove a generalization of conjectures of Hao and Schreieder \cite[Conjecture~1.7]{HS21(1)} (in the case $\kappa(X) \ge 0$) and \cite[Conjecture~1.8]{HS21(1)} assuming the existence of a good minimal model:

\begin{Lthm}[= Theorem~\ref{thm:Iitaka_decomposition_PLI}]\label{intro:thm:Iitaka_decomposition_PLI}
Let $X \to A$ be a morphism from a smooth projective variety $X$ to an abelian variety $A$. Assume $\omega_1, \dots, \omega_g \in H^0(A, \Omega_A)$ are $1$-forms such that $f^* \omega_1, \dots, f^* \omega_g$ are pointwise linearly independent. Assuming $X$ admits a good minimal model then there is a birational model $X \rat X'$ and a quotient with connected kernel $A \onto B$ to an abelian variety of dimension $\ge g$ fitting into a diagram
\begin{center}
\begin{tikzcd}[column sep = large, row sep = large]
X \arrow[r, dashed] \arrow[d] & X' \arrow[d]
\\
A \arrow[r] & B
\end{tikzcd}
\end{center}
with $X' \to B$ an isotrivial\footnote{flat with isomorphic fibers, in fact we show that $X' \to A$ is trivialized by an isogeny} smooth morphism. 
\end{Lthm}
\noindent
In particular, when $X \to A$ is smooth we can take $g = \dim{A}$ to recover \cite[Conjecture 1.8]{HS21(1)} (= Theorem~\ref{thm:birational_to_isotrivial_map}) namely that there is a birational model 
\begin{center}
\begin{tikzcd}
X \arrow[rr, dashed] \arrow[rd] & & X' \arrow[ld]
\\
& A
\end{tikzcd}
\end{center}
with $X' \to A$ isotrivial (in particular an analytic fiber bundle).

\begin{example} \label{example:Schreieder_Yang}
The birational modification in Theorem~\ref{intro:thm:Iitaka_decomposition_PLI} (and \cite[Conjecture 1.8]{HS21(1)}) is necessary due to the following example of Schreieder and Yang \cite[pp. 2]{SY22}. Let $X$ be the blowup of $E_1 \times E_2 \times \P^1$ along $E_1 \times 0 \times 0$ and $0 \times E_2 \times \infty$ where $E_1$ and $E_2$ are non-isogenous elliptic curves. Then there is a nowhere vanishing $1$-form $\omega = \pi_1^* \omega_1 + \pi_2^* \omega_2$ where $\omega_i \in H^0(E_i, \Omega_{E_i})$ is nonzero. However, both maps $X \to E_i$ are not smooth since the fiber over $0$ or $\infty$ has two irreducible components. Since the elliptic curves are non-isogenous, $X$ does not admit a smooth map to any positive dimensional abelian variety. 
\end{example}

Hao and Schreieder make a similar conjecture in general (covering the remaining case of $\kappa(X) = -\infty$).

\begin{Lconj}\cite[Conjecture 1.7]{HS21(1)} \label{conj:HS21}
Let $X$ be a smooth projective variety that admits a holomorphic $1$-form without zeros. Then $X$ is birational to a smooth projective variety $X'$ with a smooth morphism $X' \to A$ to a positive-dimensional abelian variety $A$.
\end{Lconj}

Our method is less suited to tackling this problem in general but we give partial results on the structure of such varieties in \S\ref{section:MRC}. The argument uses a slight generalization of the main theorem of \cite{PS14} that applies even in the absence of an Iitaka fibration instead supposing that certain subsheaves of $\Omega^k$ have some positivity (see Theorem~\ref{thm:generalization_of_PS14}).

Again under the assumptions of MMP, we recover a conjecture of Meng and Popa \cite[Conjecture C]{MP21} inspired by Ueno's conjecture K. 

\begin{Lthm}[= Theorem~\ref{thm:MP_conjecture_C}] \label{intro:thm:MP_conjecture_C}
Let $f : X \to A$ be an algebraic fiber space, with $X$ a smooth projective variety, $A$ an abelian variety, and general fiber $F$ with $\kappa(F) \ge 0$. Assume that $F$ admits a good minimal model. If $f$ is smooth away from a closed set of codimension at least $2$ in $A$, then there exists an isogeny $A' \to A$ such that
\[ X \times_A A' \birat F \times A' \]
i.e.\ $X$ becomes birational to a product after an \etale base change.
\end{Lthm}

Likewise, assuming abundance holds for $X$, we also verify a conjecture of Chen, Hao, and the author \cite[Conjecture B]{CCH23}.

\begin{Lthm}[= Theorem~\ref{thm:PLI_number}]\label{into:thm:PLI_number}
Let $X$ be a smooth projective good minimal model. If $g$ is the maximal number of pointwise linearly independent holomorphic $1$-forms on $X$ then there exists a smooth morphism $X \to A$ over an abelian variety of dimension $\dim{A} = g$. Moreover, $X \to A$ is isotrivial and trivialized by an isogeny.
\end{Lthm}

The main new ingredient is a decomposition result for a family of Calabi-Yau varieties equipped with a surjective map to an abelian scheme (see Theorem~\ref{thm:abelian_decomposition}). This arises from studying the decomposition theorem for $K_X$-trivial varieties in families. The proofs proceed by applying this decomposition result to the Iitaka fibration of a good minimal model.

\textbf{Conventions.} We work over the complex numbers. A \textit{variety} is an integral separated scheme of finite type over $\CC$. A \textit{minimal model} is a projective variety $X$ with terminal $\Q$-factorial singularities such that $K_{X}$ is nef. We say a minimal model is \textit{good} if $n K_X$ is base-point-free for some $n > 0$. A \textit{pair} $(X, \Delta)$ consist of a normal variety $X$ and a Weil $\Q$-divisor $\Delta = \sum a_i D_i$ on $X$ such that $K_X + \Delta$ is is $\Q$-Cartier. A \textit{flat family of pairs} $(X, \Delta) \to S$ is a flat morphism $X \to S$ and a Weil $\Q$-divisor $\Delta$ on $X$ whose components are flat over $S$ such that $(X, \Delta)_s \coloneq (X_s, \Delta_s)$ is a pair for each $s \in S$. 

\textbf{Acknowledgements.} I thank Ravi Vakil, Mihnea Popa, and Nathan Chen for many enlightening conversations and comments on this manuscript. The author would like to thank the Department of Mathematics at Harvard University for its hospitality during the 2024-2025 academic year. During the preparation of this article, the author was partially supported by an NSF Graduate Research Fellowship under grant DMS-2103099 DGE-2146755.

\section{Decomposition of the Abelian Variety Part} \label{section:decomposition}

In this section, we develop a technique for analyzing the Iitaka fibration of a variety equipped with a map $f : X \to A$ to an abelian variety. Imprecisely, we will show that if the fiber of the Iitaka fibration surjects onto $A$ then $X$ contains $A$ as a factor (really $X$ admits a rational isotrivial fibration in abelian varieties isogenous to $A$, see Theorem~\ref{thm:Iitaka_decomposition}). This decomposition is a relative version of the first stage of the (singular) Bogomolov-Beauville decomposition i.e.\ the isotriviality of the Albanese map. The essential use of this result is to provide an explicit birational model of $X$ admitting a smooth map to $A$.
\par 
In this section, we prove a direct generalization to the relative setting of this strong isotriviality (triviality after a finite \etale cover) of the Albanese map of a $K$-trivial variety with klt singulairites.

\begin{theorem} \label{thm:abelian_decomposition}
Let $g : (X, \Delta) \to S$ be a flat projective family of pairs over a locally noetherian reduced base scheme $S$ of pure characteristic zero whose fibers satisfy
\begin{enumerate}
\item $(X_s, \Delta_s)$ are klt pairs (in particular the fibers are integral with $K_{X_s} + \Delta_s$ a $\Q$-Cartier divisor) 
\item $K_{X_s} + \Delta_{X_s} \equiv_{\mathrm{num}} 0$ 
\end{enumerate}
equipped with a surjective $S$-morphism $g : X \to \cA$ where $\cA \to S$ is a polarized abelian scheme. Let $Z = f^{-1}(0_A)$ be the preimage of the zero section. Then there is an isogeny $\pi : \cB \to \cA$ equipped with a map $\rho : \cB \to \Aut^0_{(X,\Delta)/B}$ such that in the diagram
\begin{center}
\begin{tikzcd}
Z \times_S \cB \arrow[rrd, bend left] \arrow[ddr, bend right, "\sigma"] \arrow[rd, dashed, "\tilde{\sigma}"]
\\
& X \times_{\cA} \cB \arrow[r] \arrow[d] & \cB \arrow[d, "\pi"]
\\
& X \arrow[r, "f"] & \cA 
\end{tikzcd} 
\end{center}
the unique map $\tilde{\sigma} : Z \times \cB \iso X \times_{\cA} \cB$ induced by the action $\sigma : Z \times_S \cB \to X$ is an isomorphism. Hence there is an $S$-isomorphism $X \cong Z \times^G_S \cB$ where $G = \ker{(\cB \to \cA)}$.
\end{theorem}

\begin{remark}
 For $S$-schemes $X$ and $Y$ equipped with right and left $G$-actions respectively, the notation $X \times^G_S Y$ means the quotient of $X \times_S Y$ by the anti-diagonal $G$-action:
    \[ g \cdot (x,y) = (x \cdot g^{-1}, g \cdot y), \quad g \in G, \, (x,y) \in X \times Y \]
\end{remark}

The proof follows directly from the case when $S$ is a point which is essentially proved in \cite{Jinsong20} or \cite{Amb05}. The extra ingredient needed to globalize the argument is a canonical choice of isogeny $A \to \Alb_X$ trivializing the Albanese map. Then these isogenies can be glued together to an isogeny trivializing the entire family of Albanese maps\footnote{A subtlety here is that there may not be a relative Albanese map if $X \to S$ does not admit a section, indeed there is an $\Alb_{X/S}$-torsor whose nontriviality provides an obstruction to the existence of such a map.}. Hence, once the result is established over a point, we formulate the statement in families using the map $\Alb_{X/S} \to \cA$ induced by $X \to \cA$. To make this work, we need acess to global objects interpolating the Picard scheme, automorphism scheme, and Albanese of the fibers.

\subsection{The Albanese and Automorphism Schemes In Families}

Let $\Pic_{X/S} \to S$ be the fppf sheafification of the relative Picard functor \[ (T \to S) \mapsto \Pic{(X_T)} / \Pic{(T)} \]
In good situations, this is representable by a finite type group scheme and the connected components of the identity of the fibers glue to an open and closed subscheme $\Pic_{X/S}^0 \subset \Pic_{X/S}$. In such situations, we define $\Alb_{X/S} \coloneq \Pic_{\Pic_{X/S}^0/S}^0$. This satisfies a mapping property similar to that of the Albanese over a point. Given any $S$-map $f : X \to \cA$ to an abelian scheme over $S$ we can apply $\Pic_{-/S}^0$ twice to obtain a map
\[ \Alb_{X/S} \to \cA^{\vee \vee} \iso \cA \]
using the double-duality isomorphism $\cA^{\vee \vee} \iso \cA$ defined by the Poincar\'{e} bundle. This procedure produces a homomorphism
\[ \Hom_{S}(X,\cA) \to \Hom_{S}(\Alb_{X/S}, \cA) \] 
of abelian groups whose kernel is given by the maps of the form $X \to S \to \cA$ for some section $S \to \cA$. To see this final statement, note that $\Alb_{X/S} \to \cA$ factors through the zero section, $e : S \to \cA$, exactly if $\cA^{\vee} \to \Pic_{X/S}^0$ is trivial but the restriction of $\Pic^0_A$ to a subscheme is injective unless it is zero dimensionsional, so $X \to \cA$ factors through a section. Note, however, that the short exact sequence
\[ 0 \to \cA(S) \to \Hom_{S}(X, \cA) \to \Hom_{S}(\Alb_{X/S}, \cA) \to 0 \]
may not be split (at least it is not canonically so). This is a manifestation of the fact that the Albanese's universal property is correctly formulated on the category of \textit{pointed} schemes: $(X, x) \to (\Alb_X, e)$ is the initial \textit{pointed} map to an abelian variety (where abelian varieties are pointed at their origin). Indeed, when $X \to S$ does not admit a section, there may not be any ``albanese map'' $X \to \Alb_{X/S}$ over $S$. Rather, there is a $\Alb_{X/S}$-torsor of such maps which, if nontrivial, does not have a global section. Thus, if $X \to S$ admits a section $\sigma$ then the Picard functor can be rigidified along $\sigma$ giving a universal bundle on $\Pic_{X/S}^0 \times_S X$. The Albanese map $X \to \Alb_{X/S}$ is then defined by viewing the universal bundle $\cP$ on $X \times_S \Pic_{X/A}^0$ as a line bundle on $(\Pic_{X/S}^0)_X$ (the base change along $X \to S$) hence defining a map $X \to \Alb_{X/S} = \Pic_{\Pic_{X/S}^0}^0$.
To run our argument, we first must know that $\Pic_{X/S}^0$ and $\Alb_{X/S}$ are abelian schemes. We will use the following results of SGA3.

\begin{prop} \label{prop:group_scheme_properties}
Let $G \to S$ be a group scheme locally of finite presentation. Suppose that for all $s \in S$ the connected component of the identity $(G_s)^0 \subset G_s$ is smooth of locally constant dimension. Then the subfunctor $G^0 \subset G$ of sections $T \to G$ that factor through $(G_s)^0$ fiberwise is represented by an open subgroup scheme of finite presentation. Furthermore
\begin{enumerate}
\item if $S$ is reduced then $G^0 \to G$ is smooth
\item if each $(G_s)^0$ is proper and $G \to S$ is separated, then $G^0 \subset G$ is closed and $G^0 \to S$ is proper.
\end{enumerate}
\end{prop}

\begin{proof}
Open representability is given by the implication (ii) $\implies$ (iv) of \cite[Tome I, Expos\'{e} IV$_B$, Corollary 4.4]{SGA3}. (1) is also proved in \textit{loc. cit.} so we only need to prove that $G^0 \to S$ is finitely presented. Since $G^0 \subset G$ is open, $G^0 \to S$ is locally of finite presentation so it suffices to prove quasi-compactness. To do this, we pass to the case where $S$ is affine and proceed by Noetherian induction. Assuming for all proper closed subsets $T \subset S$ that $G_T^0$ is quasi-compact we must prove $G^0$ is quasi-compact. Following the arugment in \cite[Proposition~9.5.20, Lemma~9.5.1]{Kleiman}, consider affine open neighborhoods
\begin{center}
\begin{tikzcd}
G \arrow[r, "\pi"] & S
\\
U \arrow[u, hook] \arrow[r] & V \arrow[u, hook]
\end{tikzcd}
\end{center}
Since $\pi$ is smooth, it is open so $U' \coloneq \pi^{-1}(\pi(U))$ is open. Let $\alpha : G^0 \times_S G^0 \to G^0$ be the difference map $(x,y) \mapsto xy^{-1}$. It restricts to a map $\alpha' : U \times_V U \to U'$. I claim $\alpha' : U \times U \to U'$ is surjective. We will show this on fibers. For $s \in \pi(U)$, the map $\alpha'_s : U_s \times U_s \to G^0_s$ is the difference map on the irreducible $\kappa(s)$-group $G^0_s$. For any $g \in G^0_s$, the opens $U_s$ and $g \cdot U_s$ are nonempty and hence intersect say at $h \in (g \cdot U_s) \cap U_s$. This means we can write $h = g u_1$ and $h = u_2$ for $u_1, u_2 \in U_s$, but this implies $g = \alpha'_s(u_1, u_2)$ so $\alpha'_s$ is surjective. Since $U \to V$ is a map of affine schemes, $U \times_V U$ is quasi-compact, so its image $U'$ under $\alpha'$ is also quasi-compact. 
Let $T \coloneq S \sm \pi(U)$ be the closed complement. Since $\pi(U)$ is nonempty, the induction hypothesis shows $G_T^0$ is quasi-compact. However, $G^0 = U' \cup G_T$ with $U'$ a quasi-compact open and $G_T$ a quasi-compact closed subscheme, so $G^0$ is quasi-compact. 
\par 
For (2), assume each $(G_s)^0$ is proper and $G \to S$ is separated. Since $G^0 \subset G$ is open, $G^0 \to S$ is also separated. We showed above that it is finitely presented. Since $G^0_s$ is a connected group scheme over $\kappa(s)$, it is geometrically connected. Since it is also proper by assumption, we can apply \cite[IV$_3$, Corollary 15.7.10]{EGA} to conclude that $G^0 \to S$ is proper. Finally, $G^0 \to S$ is proper, and $G \to S$ is separated, so $G^0 \embed G$ is a closed embedding.
\end{proof}

\begin{lemma} \label{lemma:pic_abelian_scheme}
Let $f : X \to S$ be a flat projective map with $S$ reduced, locally noetherian, and pure characteristic zero. Assume that for all $s \in S$
\begin{enumerate}
    \item the geometric fibers $X_{\bar{s}}$ are integral,
    \item $X_{\bar{s}}$ are normal, 
    \item $X_{\bar{s}}$ has only Du Bois singularities,
\end{enumerate}
then $\Pic_{X/S}^0 \to S$ is an abelian scheme.
\end{lemma}

\begin{proof}
First, since $f : X \to S$ is flat proper with geometrically integral fibers, $\Pic_{X/S}$ is representable by a separated locally finite type group scheme over $S$ \cite[Theorem 9.4.8]{Kleiman}. Applying Proposition~\ref{prop:group_scheme_properties}, it suffices to show that each $\Pic_{X_{\bar{s}}}^0$ is smooth and proper and its dimension in locally constant on $S$. This will show that $\Pic_{X/S}^0 \to S$ exists and is a smooth proper $S$-group whose fibers are abelian varieties.
\par
Assumption (2) implies that each $\Pic_{X_{\bar{s}}}^0$ is an abelian scheme. Indeed, if $Y$ is a normal proper variety, its albanese is always an abelian variety. To see this, let $\pi : \wt{Y} \to Y$ be a resolution of singularities. Then there is a map of group schemes $\pi^* \Pic_Y^0 \to \Pic_{\wt{Y}}^0$ and we know $\Pic_{\wt{Y}}^0$ is an abelian variety since the valuative criterion for properness follows from the extensibility for line bundle from opens of regular schemes. Furthermore, since $Y$ is normal $\pi_* \struct{\wt{Y}} = \struct{Y}$, if $\L$ is a line bundle on $Y$, then $f_* f^* \L = \L \ot f_* \struct{\wt{Y}} = \L$ by the projection formula. Hence $\pi^* \Pic_Y^0 \to \Pic_{\wt{Y}}^0$ is injective, so we conclude that $\Pic_Y^0$ must also be an abelian variety. 
\par 
The local constancy of dimension will follow from assumption (3). Indeed, because the fibers have only Du Bois singularities, $R^i f_* \struct{X}$ is a vector bundle whose formation commutes with base change \cite[Lemme 1]{DJ74}. Hence $\dim{\Pic_{X_{\bar{s}}}^0} = h^1(X_{\bar{s}}, \struct{X_{\bar{s}}})$ is locally constant.
\end{proof}

Now we let $\Aut_{(X,\Delta)/S} \to S$ be the group scheme representing
\[ (T \to S) \mapsto \Aut(X_T, \Delta_T) = \{ f : X_T \iso X_T \mid f(\Delta_T) \subset \Delta_T \} \]
Sending an automorphism to its graph, we see that $\Aut_{(X,\Delta)/S} \embed \Hilb_{X \times_S X/S}$ is a locally closed subfunctor so $\Aut_{(X,\Delta)/S}$ is representable by a separated $S$-group scheme locally of finite type. Therefore, applying Proposition~\ref{prop:group_scheme_properties} immediately gives the following.

\begin{lemma} \label{lemma:aut_abelian_scheme}
Let $f : (X, \Delta) \to S$ be a flat projective family of pairs with $S$ reduced, locally noetherian, and pure characteristic zero. Assume that for all $s \in S$
\begin{enumerate}
    \item the geometric fibers $X_{\bar{s}}$ are integral,
    \item $\Aut_{(X,\Delta)_{\bar{s}}}^0$ is an abelian scheme of locally constant dimension,
\end{enumerate}
then $\Aut_{(X,\Delta)/S}^0 \to S$ is an abelian scheme.
\end{lemma}

Now we prove main result of this section. 

\begin{proof}[Proof of Theorem~\ref{thm:abelian_decomposition}]
We first prove the result when $S$ is a point. By the proof of \cite[Theorem 4.5]{Jinsong20}, $X \to \Alb_X$ is a homogeneous fibration for the $\Aut_{(X,\Delta)}^0$-action, meaning the map 
\[ \tilde{\sigma}_0 : F \times \Aut_{(X,\Delta)}^0 \to X \times_{\Alb_X} \Aut_{(X,\Delta)}^0 \]
induced by the $\Aut_{(X,\Delta)}^0$-action is an isomorphism where $F = \alb^{-1}(0)$. Therefore, if $X \to A$ is a surjective map to an abelian variety, we get a diagram
\begin{center}
    \begin{tikzcd}
        X \times_{A} \Aut_{(X,\Delta)}^0 \arrow[r] \arrow[d] & B \pullback \arrow[r] \arrow[d] & \Aut_{(X,\Delta)}^0 \arrow[d]
        \\
        X \arrow[r] \arrow[rr, bend right, "f"'] & \Alb_X \arrow[r] & A
    \end{tikzcd}
\end{center}
where $B$ is an abelian variety defined by the fiber product. Let $C = \ker{(\Alb_X \to A)}$ and $\wt{C} = r^{-1}(C)$. The translation action $\Aut_{(X,\Delta)}^0 \acts \Alb_X$ induces a map $r : \Aut_{(X,\Delta)}^0 \to \Alb_X$, giving a section of $\pi : B \to \Aut_{(X,\Delta)}^0$, and thus a decomposition $B \cong \Aut_{(X,\Delta)}^0 \times C$. Consider the diagram
\begin{center}
\begin{tikzcd}
    X \times_{\Alb_X} (\Aut_{(X,\Delta)}^0 \times C) \arrow[r] & X \times_{\Alb_X} B \arrow[r, equals] & X \times_{A} \Aut_{(X,\Delta)}^0
    \\
    F \times \Aut_{(X,\Delta)}^0 \times \wt{C} \arrow[u] \arrow[rr] & & Z \times \Aut_{(X,\Delta)}^0 \arrow[u, "\sigma"]
\end{tikzcd}
\end{center}
where the map $\sigma : Z \times \Aut_{(X,\Delta)}^0 \to X \times_{A} \Aut_{(X,\Delta)}^0$ is $(z, a) \mapsto (a \cdot z, a)$, which is clearly injective. Furthermore, $F \times \Aut_{(X,\Delta)}^0 \times \wt{C} \to X \times_{\Alb_X} (\Aut_{(X,\Delta)}^0 \times C)$ is given by $(f, a, c) \mapsto (a \cdot f, a + c, -r(c))$, and $F \times \Aut_{(X,\Delta)}^0 \times \wt{C} \to Z \times \Aut_{(X,\Delta)}^0$ is $(f, a, c) \mapsto (c \cdot f, a)$. Furthermore, since $\tilde{\sigma}_0$ is an isomorphism, $F \times \Aut_{(X,\Delta)}^0 \times \wt{C} \to X \times_{\Alb_X} (\Aut_{(X,\Delta)}^0 \times C)$ is surjective. Hence $Z \times \Aut_{(X,\Delta)}^0 \to X \times_A \Aut_{(X,\Delta)}^0$ is also surjective proving it is an isomorphism.
\par 
Now we prove the result over an arbitrary base $S$. By the main result of \cite{KK10}, the fibers $X_{\bar{s}}$ have Du Bois singularities so we are in the situation of the lemmas.
Furthermore, the proof of \cite[Theorem 4.5]{Jinsong20} shows that the automorphism schemes $\Aut_{(X,\Delta)_{\bar{s}}}^0$ are all abelian varieties of dimension $q = h^1(X_s, \struct{X_s})$, which is locally constant because of Du Bois singularities by \cite[Corollary 1.2]{KK10}. Hence applying Lemma~\ref{lemma:pic_abelian_scheme} and Lemma~\ref{lemma:aut_abelian_scheme} shows that $\Alb_{X/S}$ and $\Aut_{(X,\Delta)/S}^0$ are abelian schemes over $S$. Furthermore, there is an isogeny $\Aut_{(X,\Delta)/S}^0 \onto \Alb_{X/S}$, given by the action of $\Aut_{(X,\Delta)/S}^0 \acts \Alb_{X/S}$ by translations induced by the canonical action on $X$ and functoriality of $\Alb_{-/S}$. Consider the analogous diagram relative over $S$
\begin{center}
    \begin{tikzcd}
        Z \times_S \Aut_{(X,\Delta)/S}^0 \arrow[rrd, bend left] \arrow[ddr, bend right] \arrow[rd, dashed]
        \\
        & X \times_{\cA} \Aut_{(X,\Delta)/S}^0 \arrow[d] \arrow[r] & \Aut_{(X,\Delta)/S}^0 \arrow[d]
        \\
        & X \arrow[r] & \cA
    \end{tikzcd}
\end{center}
The case over a point proves that on each fiber over $s \in S$ the dotted map becomes an isomorphism. Since $Z \times_S \Aut_{(X,\Delta)/S}^0 \to S$ and $X \times_{\cA} \Aut_{(X,\Delta)/S}^0 \to S$ are flat and proper (since they are fiber products of flat and proper maps), the dashed arrow is an isomorphism by \cite[III$_1$, Proposition 4.6.7(i)]{EGA}. Since $X \to S$ is projective, the relative Hilbert scheme, and hence $\Aut_{(X,\Delta)/S}^0$ are projective over $S$. Therefore $\Aut_{(X,\Delta)/S}^0 \to S$ is a polarizable abelian scheme. Using suitable polarizations to form complements, there is a map $\cA \to \Aut_{(X,\Delta)/S}^0$ such that the composition
\[ \cA \to \Aut_{(X,\Delta)/S}^0 \to \cA \]
is an isogeny $\varphi$. Then the composition gives us a similar diagram
\begin{center}
    \begin{tikzcd}
        Z \times_S \cA \arrow[rrd, bend left] \arrow[ddr, bend right] \arrow[rd, dashed]
        \\
        & X \times_{\varphi} \cA \arrow[d] \arrow[r] & \cA \arrow[d, "\varphi"]
        \\
        & X \arrow[r] & \cA
    \end{tikzcd}
\end{center}
and the dashed map $Z \times_S \cA \to X \times_{\varphi} \cA$ is the pullback of the dashed map $Z \times_S \Aut_{(X,\Delta)/S}^0 \to X \times_{\cA} \Aut_{(X,\Delta)/S}^0$ along $\id_X \times (\cA \to \Aut_{(X,\Delta)/S}^0)$, and hence is also an isomorphism, so we win. 
\end{proof}

\begin{remark}
It is important that the proof only uses the Albanese map in the special case $S = *$. We remarked earlier that an Albanese map may not exist over $S$ except after a finite \etale cover $S' \to S$ to trivialize a $\Alb_{X/S}$-torsor. Passing to such a cover first would not be sufficient for our applications to the flat locus of the Iitaka fibration $X^{\min} \to S$. Indeed, we are allowed to shrink only to a Zariski open if we hope to obtain an actual birational model of $X$. Understanding some \etale open would not suffice.
\end{remark}

\section{Application to the Iitaka Fibration} \label{section:iitaka}

It will be of particular interest to apply the decomposition results of the previous section to the Iitaka fibration $X^{\min} \to S$ of a good minimal model of $X$ (assuming one exists). Indeed, whenever $\psi : X \to S$ is such that $K_X = \psi^* L$, an application of Theorem~\ref{thm:abelian_decomposition} immediately gives the following generalization of \cite[Theorem 4.1]{CCH23}:

\begin{thm} \label{thm:Iitaka_decomposition}
Let $\psi : X \to S$ be a proper morphism of normal varieties with $X$ having klt $\Q$-gorenstein singularities and $K_X = \psi^* L$. Let $f : X \to A$ be a map to an abelian variety such that the image of the general fiber of $\psi$ in $A$ has dimension $d$. Then there exists a quotient with connected kernel $q : A \to B$ to an abelian variety $B$ of dimension $d$ and a birational map $X \birat Z \times^G B'$ making the diagram
\begin{center}
    \begin{tikzcd}
    X \arrow[d, dashed] \arrow[r, "f"] & A \arrow[d, "q"] 
    \\
    Z \times^G B' \arrow[r] & B 
    \end{tikzcd}
\end{center}
commute. Here, $B' \to B$ is an isogeny with kernel $G$ and $Z$ is a smooth projective variety with a $G$-action.
\end{thm}

\begin{proof}
Let $q : A \to B$ be the quotient as in Lemma~\ref{lemma:dimesnion_and_PLI_forms} determined by the image of $F \to \Alb_F \to A$.
Let $U \subset S$ be the locus where $\psi$ is flat fibers have at worst klt singularities (which is open by \cite[Lemma~5.17]{KM98}). Then we can apply Theorem~\ref{thm:abelian_decomposition} to the diagram,
\begin{center}
    \begin{tikzcd}
        X_U \arrow[rr, "q \circ f"] \arrow[rd, "\psi"] & & B \times U \arrow[ld]
        \\
        & U
    \end{tikzcd}
\end{center}
so we get an isogeny $B' \to B$ (the Theorem gives an abelian scheme with an isogeny to $B \times U$ but this is a constant abelian scheme by \cite[Lemma 3.2]{CCH23} -- along with an isomorphism $X_U \iso B' \times^G Z$ where $G \coloneq \ker{(B' \to B)}$ making the requisite diagrams commute. Now following the argument of \cite[Theorem 4.1]{CCH23}, we choose a smooth equivariant compactification $Z \embed \ol{Z}$, meaning $\ol{Z}$ is smooth and projective with a $G$-action, and $Z \embed \ol{Z}$ is an $G$-equivariant open embedding. This relies on an equivariant version of Nagata compactification, which holds only for finite group actions. Given a finite group action $G \acts X$, consider the scheme-theoretic image of $X \to \prod_{g \in G} \overline{X}$ given by applying $g$. Then one can embed $X \embed \overline{X}$ for any given Nagata compactification \cite{Nagata63} and resolve equivariantly using the existence of a functorial resolution of singularities \cite[Prop.~3.9.1]{Kollar09}. From the equivariant compactification above, we have birational maps
\[ X \birat X_U \iso B' \times^G Z \embed B' \times^G \ol{Z} \]
compatible with the projections to $A$ and $B$. Furthermore, $B' \times^G \ol{Z}$ is a smooth projective variety since $Z$ is smooth projective and $G \acts B'$ freely.
\end{proof}

\begin{remark}
Since $X$ was not required to be complete we can apply the result not only to the Iitaka fibration of a good minimal model $X^{\min} \to S$ but to any model of the Iitaka fibration $\psi : X' \rat S'$ which is \textit{almost holomorphic}, meaning the general fiber is contained in the defined locus of $\psi$ (i.e. it restricts to a \textit{proper} map $X'_U \to U$ over some open) and $K_X = \psi^* L + E$ with $\Supp{E}$ not dominating $S$. In all our subsequent results, ``admitting a good minimal model'' can be replaced by the \textit{a priori} weaker condition of having a model as above. It is not clear to the author if the existence of such a model is easier than the existence of good minimal models. 
\end{remark}

For ease of exposition, we recall the main notation and construction of \cite[\S2]{CCH23}. Let $X$ be a smooth projective variety with Kodaira dimension $\kappa(X) \geq 0$ equipped with a map $f : X \to A$ to an abelian variety.

\begin{defn}
Let $X$ be a smooth projective variety and $f : X \to A$ a morphism to an abelian variety. We say that $f : X \to A$ satisfies $(\ast)_g$ for some integer $g \ge 1$ if there is a subspace $W \subset H^0(A, \Omega_A)$ with $\dim{W} = g$ of $1$-forms such that $Z(f^* \omega) = \varnothing$ for any nonzero $\omega \in W$.
\end{defn}

\begin{rmk}
This is exactly the pointwise linearly independent (PLI) condition considered in \cite[\S2]{CCH23}: namely the existence of a sequence of holomorphic $1$-forms $\omega_1, \dots, \omega_g \in H^0(A, \Omega_A)$ such that $f^* \omega_1, \dots, f^* \omega_g$ are pointwise linearly independent. This is also equivalent to the existence of a vector sub-bundle $\cO_X^{\oplus g} \subset\Omega_X^{1}$ factoring through $f^* \Omega_A \to \Omega_X$.
\end{rmk}

Our goal is to produce a candidate surjection from $X$ onto an abelian variety of dimension $\ge g$. Let $\phi : X \rat X^{\can}$ be the Iitaka fibration of $X$, and consider a proper birational model $X \to S$ resolving the indeterminacies of $X \rat X^{\can}$ with $S$ normal. Assume that $X$ has only $\Q$-Gorenstein rational singularities\footnote{We will apply this construction in two cases: (1) to a good minimal model $X^{\min}$ of $X$ which has $\Q$-factoral terminal singularities and $X^{\min} \to S$ is the morphism induced by the semiample divisor $K_{X^{\min}}$. Or $S$ is a desingularization of $X^{\can}$ and $X' \to S$ is a smooth resolution of the rational map $X \rat S$.}. The morphism $f : X \to A$ induces a rational map $X' \to A$, which is everywhere defined since $A$ is an abelian variety and $X'$ has at worst rational singularities. These fit into the diagram below:
\begin{center}
\begin{tikzcd}[row sep=small]
& F \arrow[rr] \arrow[d] & & \Alb_{F} \arrow[d]
\\
X \arrow[dd, dashed] & X' \arrow[l] \arrow[rr] \arrow[dd] & & A \arrow[dd] \arrow[dl]
\\
&  & Q_X \arrow[rd] &
\\
X^{\can} & S \arrow[l] \arrow[ru, dashed] \arrow[rr] & & \Alb_S
\end{tikzcd}
\end{center}

\noindent In the diagram, $F$ denotes a general fiber of $X' \rightarrow S$ and $Q_X$ is the cokernel of $\Alb_F \to A$. Note the general fiber also has rational singularities, so there is no ambiguity about the meaning of $\Alb_F$. This gives the third column of induced morphisms on Albanese varieties. Since the fiber $F$ is contracted in $X' \to Q_X$ by definition, the map $X' \to Q_X$ factors birationally through $X' \to S$ by rigidity. If $S$ has rational singularities, $S \rat Q_X$ extends to a morphism but we will not need this.

\begin{lemma}[{c.f. \cite[Proof of Conjecture 1.2]{PS14}}] \label{lemma:dimesnion_and_PLI_forms}
Following the notation above, the inequalities hold:
\begin{equation}\label{eq:1}
    \dim{F} \ge \dim{A} - \dim{Q_X} \ge \dim{W}.
\end{equation}
and there is a quotient $A \onto B$ with connected kernel to an abelian variety of dimension $\dim{B} \ge \dim {W}$ such that $F \to B$ is surjective.
\end{lemma}

\begin{proof}
The first inequality follows from the fact that $F \onto \Alb_F$ is surjective (in fact a contraction) \cite[Thm.~1]{Kawamata81} and $\dim{Q_X} \ge \dim{A} - \dim{\Alb_F}$ (by definition of $Q_X$). For the second inequality, note that the composition morphism $f \colon X \to \Alb_{X} \to Q_X$ birationally factors through the Iitaka fibration on $X$. Therefore, we may apply \cite[Thm.~2.1]{PS14} to show that
\[ W \cap f^{*} H^{0}(Q_X, \Omega_{Q_X}^1) = \{ 0 \}, \]
which implies that $\dim W + \dim Q_X \leq \dim \Alb_{X}$. Note that $f^{\ast}$ is injective since $A \to Q_X$ is smooth and surjective.
\par 
Varying the fiber $F$, rigidity implies that the images of $F \to A$ are translates of a fixed abelian subvariety, say $B_0 \subset A$. Fixing a polarization on $A$ and dualizing yields a morphism
\[ q : A \to A^\vee \onto B \coloneq B_0^\vee \]
such that $B_0 \subset A$ maps surjectively onto $B$. We claim that $F \to B$ is surjective, which follows from the fact that $F \to B_0$ is surjective. Note that $Q_X = \Alb_X / B_0$ so 
\[ \dim{B} = \dim{B_0} = \dim{\Alb_X} - \dim{Q_X} \ge \dim{W} \]
Of course, by Stein factorization, we can assume that $A \to B$ has connected fibers. 
\end{proof}

Combining Theorem~\ref{thm:Iitaka_decomposition} with Lemma~\ref{lemma:dimesnion_and_PLI_forms} we obtain a generalization of the main result of \cite{CCH23}.

\begin{thm} \label{thm:Iitaka_decomposition_PLI}
Let $X$ be a smooth projective variety admitting a good minimal model and $f : X \to A$ be a map to an abelian variety satisfying property $(\ast)_g$. Then there exists a quotient with connected kernel $q : A \to B$ to an abelian variety $B$ of dimension $\ge g$ and a birational map $X \birat Z \times^G B'$ making the diagram
\begin{center}
    \begin{tikzcd}
    X \arrow[d, dashed] \arrow[r, "f"] & A \arrow[d, "q"] 
    \\
    Z \times^G B' \arrow[r] & B 
    \end{tikzcd}
\end{center}
commute. Here, $B' \to B$ is an isogeny with kernel $G$, and $Z$ is a smooth projective variety with a $G$-action.
\end{thm}

\begin{proof}
Let $\psi : X^{\min} \to S$ be the Iitaka fibration of some good minimal model.
By Lemma~\ref{lemma:dimesnion_and_PLI_forms} we get a quotient with connected kernel $q : A \to B$ such that the fiber $F$ of $\psi$ maps surjectively onto $B$ and $\dim{B} \ge g$. Therefore, we apply Theorem~\ref{thm:Iitaka_decomposition} and conclude.
\end{proof}

\section{Proofs of the Main Results}

Now we have all the ingredients to prove the main results.

\subsection{Conjecture 1.8 of \cite{HS21(1)}}

Assuming the abundance conjecture and termination of flops we prove a conjecture of Hao and Schreieder \cite[Conjecture 1.8]{HS21(1)}.

\begin{thm}\label{thm:birational_to_isotrivial_map}
Let $X \to A$ be a smooth morphism from a smooth projective variety $X$ to an abelian variety $A$. If $\kappa(X) \ge 0$ and assuming $X$ admits a good minimal model then there is a birational model
\begin{center}
\begin{tikzcd}
X \arrow[rr, dashed] \arrow[rd] & & X' \arrow[ld]
\\
& A
\end{tikzcd}
\end{center}
with $X' \to A$ an isotrivial\footnote{flat with isomorphic fibers} smooth projective morphism (i.e. an analytic fiber bundle).
\end{thm}

\begin{proof}
We apply Theorem~\ref{thm:Iitaka_decomposition_PLI} to $X \to A$ to conclude that $X \birat Y \times^G A'$, where $A' \to A$ is an isogeny. Let $\wt{Y} \to Y$ be a $G$-equivariant resolution of singularities. Then $X' = \wt{Y} \times^G A'$ gives the requisite birational model with a smooth morphism $X' \to A'/G = A$
\end{proof}

\subsection{Conjecture C of \cite{MP21}}

Assuming MMP and abundance we also prove Conjecture C of \cite{MP21}

\begin{thm} \label{thm:MP_conjecture_C}
Let $f : X \to A$ be an algebraic fiber space, with $X$ a smooth projective variety, $A$ an abelian variety, and general fiber $F$ with $\kappa(F) \ge 0$. Assume that $F$ admits a good minimal model. If $f$ is smooth away from a closed set of codimension at least $2$ in $A$, then there exists an isogeny $A' \to A$ such that
\[ X \times_A A' \birat F \times A' \]
i.e.\ $X$ becomes birational to a product after an \etale base change.
\end{thm}

\begin{proof}
By \cite{Lai11} $X$ admits a good minimal model $X^{\min}$. Then it suffices to show that $X^{\min} \to A \times X_{\can}$ is surjective so when we apply Theorem~\ref{thm:Iitaka_decomposition} the resulting quotient $A \to B$ is an isomorphism. There are dominant rational maps
\[ X \rat X^{\min} \rat X_{\can, f} \rat X_{\can} \times A \]
But by \cite[Proof of Corollary~D]{MP21} there is an isogeny $A' \to A$ such that 
\[ X_{\can, f} \times_A A' \cong F_{\can} \times A' \cong X_{\can} \times A' \]
and $X^{\min} \rat X_{\can, f}$ is dominant so we conclude. Therefore Theorem~\ref{thm:Iitaka_decomposition} applies to give $X^{\min} \cong Z \times^G A'$, so choosing a resolution $\wt{Z} \to Z$ gives the requisite birational model $X' \cong \wt{Z} \times^G A'$, which has a smooth map $X' \cong Z \times^G A' \to A'/G = A$ and $X' \times_A A' \cong Z \times A'$.
\end{proof}

\subsection{Conjecture B of \cite{CCH23}}

Assuming the abundance conjecture, we also verify a conjecture of Chen, Hao, and the author \cite[Conjecture~B]{CCH23}.

\begin{thm} \label{thm:PLI_number}
Let $X$ be a smooth projective good minimal model. If $g$ is the maximal number of pointwise linearly independent $1$-forms on $X$, then there exists a smooth morphism $X \to A$ over an abelian variety of dimension $\dim{A} = g$.
\end{thm}

In this section, we first prove a direct generalization of \cite[Theorem 4.5]{CCH23} that implies Theorem~\ref{thm:PLI_number}.

\begin{theorem} \label{thm:factoringthroughAV}
Let $X$ be an $n$-dimensional smooth projective good minimal model, and let $f \colon X \rightarrow A$ be a morphism to an abelian variety $A$. Then the following are equivalent:
\begin{enumerate}
    \item $f : X \to A$ satisfies $(\ast)_g$;
    \item $X$ admits a smooth morphism $\varphi \colon X \rightarrow B$ to an abelian variety $B$ of dimension $\ge g$, and there is a quotient map $q : A \to B$ with connected kernel such that $\varphi$ fits into the commutative diagram
    \begin{center}
    \begin{tikzcd}
    X \arrow[d, swap, "\varphi"] \arrow[r, "f"] & A \arrow[ld, "q", two heads] 
    \\
    B
    \end{tikzcd}
    \end{center}
\end{enumerate}
When either of these holds, there is an isomorphism $X \iso B' \times^G Z$ compatible with $\varphi$, where $B' \to B$ is an isogeny with kernel $G$, and $Z$ is a smooth minimal model with a $G$-action.
\end{theorem}

\begin{proof}
(2) $\implies$ (1) is immediate. Conversely, suppose (1) holds for $X$. By Theorem~\ref{thm:Iitaka_decomposition_PLI} applied to $f : X \to A$, there exists a quotient $q : A \to B$ with $\dim{B} \ge g$ and a birational map
\[ X \birat B' \times^G Z \]
where $B' \rightarrow B$ is an isogeny with kernel $G$. From here, we repeat the proof of \cite[Theorem 4.5]{CCH23}.
\par 
Since $X \birat B' \times^G Z$, there is an \'etale $G$-cover $X'$ of $X$ admitting a $G$-equivariant birational map $\alpha\colon X' \to B' \times Z$, where $B'\to B$ is an isogeny with kernel $G$. Moreover, $\alpha$ descends to a birational map $\overline{\alpha}\colon X\dashrightarrow B'\times^G Z$. Since $Z$ was constructed as a $K$-trivial fibration over the Iitaka base $S$, which is of log general type, by the main results \cite{Lai11}, \cite{BCHM10}, and \cite{Prokhorov21}, we can run a $G$-equivariant MMP with scaling along a $G$-invariant ample divisor to get a minimal model $Z^{\min}$ of $Z$ with a compatible $G$-action. Hence we get a birational map between two minimal models $X \dashrightarrow B' \times^G Z^{\min}$, which is a composition of flops by \cite{Kawamata08}. In particular, we have an isomorphism between codimension one open subsets of $X$ and $B' \times^G Z^{\min}$. Let $U$ denote the largest common open subvariety. Therefore the following diagram commutes
    \begin{center}
    \begin{tikzcd}
    X' \arrow[rr, dashed, "\alpha"] \arrow[d, "p_1"] && B' \times Z^{\min} \arrow[d, "p_2"]
    \\
    X \arrow[rr, dashed, "\overline{\alpha}"] && B' \times^G Z^{\min} 
    \\
    & U \arrow[ul, hook, "i_1"] \arrow[ur, hook, "i_2"']
    \end{tikzcd}
    \end{center}
Note that $X'$ and $B' \times^G Z^{\min}$ are birational minimal models. Hence $\alpha$ is again a composition of flops and $p_1^{-1}(U)$ is isomorphic to $p_2^{-1}(U)$ via $\alpha$ since $p_i$ are \etale maps. An argument in the proof of \cite[Thm.~5.10]{HS21(1)} (c.f. \cite[Thm.~3.4]{Hao23}) shows that all flops of $B' \times Z^{\min}$ arise from flops of $Z^{\min}$. Hence $\alpha : p_1^{-1}(U) \iso p_2^{-1}(U) = B' \times V$ for some open $V \subset Z^{\min}$, using the fact that $U$ is chosen to be maximal among opens over which $\overline{\alpha}$ is defined. From the description of $p_2$, the $G$-action on $A'\times V$ is diagonal. Hence there exists a sequence of $G$-equivariant flops $Z^{\min} \birat Z^+$ such that $X' \iso B' \times^G Z^+$. Since $X'$ is smooth, $Z^+$ must be smooth as well. Finally, smoothness of the map $\varphi : X \to B$ is immediate from the commutativity of the diagram
\begin{center}
    \begin{tikzcd}
        X \arrow[rd, "\varphi"'] \arrow[r, "\sim"] & B' \times^G Z \arrow[d, "\pi_1"]
        \\
        & B
    \end{tikzcd}
\end{center}
and smoothness of $\pi_1$ which follows from the fact that $G$ acts freely on $B'$.
\end{proof}

\begin{proof}[Proof of Theorem~\ref{thm:PLI_number}]
We apply Theorem~\ref{thm:factoringthroughAV} to the Albanese $a : X \to \Alb_X$. Therefore, we have a smooth morphism $X \to B / G$. If $\dim{B} = \dim{B/G} > g$, then there is a frame of PLI forms on $X$ pulled back from $B/G$ contradicting the maximality of $g$. Hence $g = \dim{B}$ so we conclude.
\end{proof}

\section{Extension to the Uniruled Case} \label{section:MRC}

Recall the conjecture of Hao and Schreieder, Conjecture~\ref{conj:HS21} in this text, predicts whenever $X$ admits a nonvanishing $1$-form it is birational to a smooth morphism over some abelian variety. When $X$ admits a good minimal model, our Theorem~\ref{intro:thm:Iitaka_decomposition_PLI} resolves this. However, this necessarily only applies when $\kappa(X) \ge 0$. In this section, we extend our results somewhat to the case of Kodaira dimension $-\infty$ by incorporating the maximally rationally connected (MRC) fibration into the picture. Recall that any uniruled variety $X$ has a unique almost-holomorphic fibration $X \rat Y$, well-defined up to birational equivalence, such that $Y$ is not uniruled and the general fiber is rationally connected \cite[Section~IV.5]{KollarRatCurves}. The following theorem is the analog of Theorem~\ref{thm:Iitaka_decomposition_PLI} allowing for an extension by a rationally connected fibration. From here, to further prove Conjecture~\ref{conj:HS21}, one would need to understand bad reduction in the MRC fibration, which is out of reach using the tools developed here.

\begin{theorem} \label{thm:main_MRC}
Let $X$ be a smooth projective variety equipped with a map $f : X \to A$ to an abelian variety satisfying $(\ast)_g$. Assume the base $Y$ of the MRC fibration $X \rat Y$ admits a good minimal model. Then there exists a quotient with connected kernel $q : A \to B$ to an abelian variety $B$ of dimension $\ge g$ and a birational map $Y \rat Z \times^G B'$ making the diagram
\begin{center}
    \begin{tikzcd}
        X \arrow[r, dashed] \arrow[rr, bend left, "f"] & Y \arrow[d, dashed] \arrow[r] & A \arrow[d, "q"]
        \\
        & Z \times^G B' \arrow[r] & B
    \end{tikzcd}
\end{center}
commute. Here, $B' \to B$ is an isogeny with kernel $G$, and $Z$ is a smooth projective variety with a $G$-action. 
\end{theorem}

If the MRC fibration were a morphism, this would be an obvious consequence of Theorem~\ref{thm:Iitaka_decomposition_PLI} since the $1$-forms verifying $(\ast)_g$ are pulled back through $X \to Y \to A$ and hence also satisfy $(\ast)_g$ with respect to $Y \to A$. However, for a rational map $X \rat Y$ (even an almost holomorphic one) it is not clear if we can conclude $(\ast)_g$ for $Y \to A$. To prove the above result, we extract slightly more information from the proof of \cite{PS14} that, even for a \textit{rational} map $X \rat Y$, will transport cohomological information on $Y$ upwards to information on the zero locus of forms on $X$. First we prove a direct generalization of the main theorem of \cite{PS14}.

\begin{theorem} \label{thm:generalization_of_PS14}
Let $f : X \to A$ be in $\Var_A$. Consider the sheaf consisting of $k$-forms killed by $-\wedge f^* \omega$ for all $\omega \in H^0(A, \Omega_A^1)$ i.e.
\[ P\Omega_X^k \coloneq \ker{(\Omega_X^k \to \Omega_X^{k+1} \ot H^0(A, \Omega_A^1)^\vee)} \]
where the map is given by wedge with $1$-forms pulled back from $A$.
Suppose there is a line bundle $\cN \embed P\Omega_X^k$ and an ample $\L \in \Pic(A)$ so that $H^0(X, \cN^{\ot d} \ot f^* \L^{-1}) \neq 0$ for some $d \ge 1$. Then every $\omega \in H^0(A, \Omega_A^1)$ satisfies $Z(f^* \omega) \neq \varnothing$.
\end{theorem}

\subsection{Generic Nonvanishing Along Rational Maps}

Here we recall some ingredients introduced in \cite{PS14}. In what follows, we notate $n = \dim{X}$ and $g = \dim{A}$.

\begin{defn}
Let $A$ be an abelian variety and $\Var_A$ the slice category of smooth projective varieties over $A$. For any $f : X \to A$ in $\Var_A$, there is a natural complex $C_X$ in $D^{\flat}_{\mathrm{Coh}}(f^* T^* A)$
\[ C_X \coloneq \Big[ \pi^* \struct{X} \to \pi^* \Omega_X^1 \to \pi^* \Omega_X^2 \to \cdots \to \pi^* \Omega_X^n \Big] \]
supported in degrees $[-n, 0]$, where $\pi : f^* T^* A \to X$ denotes the projection. The differential is induced by the adjoint of pullback: $\struct{X} \to \Omega_X^1 \ot f^* \T_A$. Concretely, given a basis $\omega_1, \dots, \omega_g \in H^0(A, \Omega_A)$ and letting $s_1, \dots, s_g \in H^0(A, T_A)$ denote the dual basis; the differential is given by the formula
\[ \Omega^p_X \ot \Sym_{n}{(\T_A)} \to \Omega^{p+1}_X \ot \Sym_{n+1}{(\T_A)} \quad \quad \theta \ot s \mapsto \sum_{i = 1}^g (\theta \wedge f^* \omega_i) \ot s_i s. \]
where $\omega_1, \dots, \omega_g \in H^0(A, \Omega_A)$ form a basis and $s_1, \dots, s_g \in H^0(A, T_A) = H^0(A, \Omega_A)^\vee$ denotes the dual basis. Notice that this differential is compatible with the grading on $\struct{f^* T^* A}$ if we apply proper shifts to give a $\Gm$-equivariant structure:
\[  \struct{T^* A}(-g) \to \pi^* \Omega_X^1 \ot \struct{f^* T^* A}(1-g) \to \pi^* \Omega_X^2 \ot \struct{f^* T^* A}(2-g) \cdots \to \Omega_X^n \ot \struct{f^* T^* A}(n-g) \]
\end{defn}

\begin{rmk}
Since $T^* A = A \times V$, where $V = H^0(A, \Omega_A^1)$, we will often identify $C_X$ with the complex of coherent graded $\struct{X} \ot S_{\bullet}$-modules on $X$ obtained by taking the second projection
\[ C_{X, \bullet} = p_{2*} C_X = \Big[ \struct{X} \ot S_{\bullet -g} \to \Omega_X^1 \ot S_{\bullet - g + 1} \to \cdots \to \Omega_X^n \ot S_{\bullet - g + n} \Big] \]
where $S_\bullet \coloneq \Sym_{\bullet}{(V^\vee)}$ is the graded algebra of global multi-tangent fields on $A$.
\end{rmk}

\noindent
$C_X$ has two important interpretations that motivate its study. First, the dual (shifted by $[2 \dim{X}])$ of the Kozul resolution for the structure sheaf of the zero section of $\pi : T^* X \to X$ is the complex induced by the tautological section of $\pi^* \Omega_X^1$
\[ \pi^* \struct{X} \to \pi^* \Omega_X^1 \to \pi^* \Omega_X^2 \to \cdots \to \Omega_X^n \]
meaning $C_X$ is the pullback of this complex under $\d{f} : f^* T^* A \to T^* X$. Hence (c.f.\ \cite[Lemma~14.1]{PS14})
\[ \Supp{C_X} = Z_f = \{ (x, \omega) \in X \times V \mid (f^* \omega)(T_x X) = 0 \} \} \]
Moreover, $C_X$ computes the associated graded of the direct image for the trivial Hodge module $(\struct{X}, F)$ 
\[ \gr^F_\bullet f_+ (\struct{X}, F) = \Rbf f_* C_{X, \bullet} \]
Furthermore, the cohomology ``commutes'' with $\gr^F_\bullet$ so that $C_{X, \bullet}$ also computes the graded parts of the decomposition of $f_+(\struct{X}, F)$
\[ \gr^F_\bullet \cH^i f_+ (\struct{X}, F) = \Rbf^i f_* C_{X, \bullet} \]
for each $i$ by \cite[Proposition 2.11]{PS13}. Taking the support inside $T^* A$ recovers Kashiwara's estimate on the characteristic variety: $\mathrm{Ch}(f_{+} \struct{X}) = \Supp{\left(\Rbf (f \times \id_V)_* C_X\right)} \subset (f \times \id_V)(Z_f)$.
Towards applying generic vanishing results (c.f. \cite[Propostion 10.2]{PS14} and the strategy of \cite{Villadsen21}) write $C^\alpha_{X, \bullet} \coloneq C_{X, \bullet} \ot f^* \alpha$ for $\alpha \in \Pic^0_A$. We will often consider the sheaves $R^i p_{2*} C_{X,\bullet}^\alpha$ on $V$, where $p_2 : X \times V \to V$ is the projection. Varying over $\alpha$ these can be packaged together into a complex of sheaves on $\hat{A} \times V$ obtained by the Fourier-Mukai transform:
\[ E_X = \Rbf \Phi_P(\Rbf f_* C_{X}) = \Rbf (p_{23})_* \left( [p_{13}^* \Rbf f_* C_X] \ot^{\LL} p_{12}^* \cP \right) \]
using the projections on $A \times \hat{A} \times V$ and the Poincar\'{e} bundle $\cP$ on $A \times \hat{A}$. Since $p_{12}^* \cP$ is flat for $p_{23}$, formation of this complex commutes with arbitrary base change. In particular, for $g_{\alpha} : \{ \alpha \} \times V \to \hat{A} \times V$ we get,
\[ \Lbf g_{\alpha}^* E_X = \Rbf (p_{2})_* \left( \Rbf f_* C_X \ot^{\Lbf} p_1^* \alpha \right) = \Rbf (p_{2})_* \Rbf f_* \left( C_X \ot^{\Lbf} f^* \alpha \right) = \Rbf p_{2*} C_X^\alpha \] 
where $p_1, p_2$ are the projections of $A \times V$. Generic vanishing results \cite[Proposition 4.15]{PS13} or Simpson's description of the cohomology jump loci on the moduli space of rank 1 Higgs bundles shows that $\Lbf g_\alpha^* E_X$ has locally free cohomology sheaves for general $\alpha \in \Pic^0_A$. 
\par
Therefore, $\H^i(E_X)$ are locally free away from codimension $2$ on $\hat{A} \times V$ (see \cite[Theorem 1.8]{PS13} and \cite[Proof of Proposition 10.2]{PS14}) and with support contained in $\hat{A} \times Z_f$. Hence to show that $\Supp{E_X}$ dominates $V$ it would suffice to show that any $\H^i(E_X)$ is nonzero at the generic point. We cannot quite do this, instead a cyclic covering construction will produce a Hodge module $\cM$ on $A$ whose $\gr_\bullet^F \cM$ is computed as above so its Fourier-Mukai transform has the same local freeness property. Then we produce a nonzero graded submodule $\F \subset \gr_\bullet^F \cM$ supported on $\hat{A} \times Z_f$ which is enough to conclude \cite[Proposition 10.2]{PS14}.

The following lemma illustrates some advantageous properties of $(R^{k-n} p_{2*} C_X^\alpha)$ that lead to the method of \cite{PS14} working for all $k$.

\begin{lemma} \label{lemma:dominant_morphism_functoriality}
Let $\phi :  X' \to X$ be a dominant (resp.\ birational) morphism in $\Var_A$. The induced map,
\[ C_{X}[n'-n] \to \Rbf \phi_* C_{X'} \]
induces an injection (resp.\ isomorphism) for each $k \ge 0$ and $\alpha \in \Pic^0_A$ 
\[ (R^{k-n} p_{2 *} C_{X}^\alpha)_{g-k} \to (R^{k-n'} p_{2 *} C_{X'}^\alpha)_{g-k} \]
of the graded parts in degree $g-k$ where $n' = \dim{X'}$ and $n = \dim{X}$.
\end{lemma}

\begin{proof}
The adjoint of the natural map is induced by pullback on differential forms $\phi^* \Omega_X \to \Omega_{X'}$. To be consistent with the cohomological degrees of $C_X$ (since $\Omega_X^k$ is placed in degree $-(n-k)$) we must shift by $(n' - n)$. However, no shift is needed in the ``cotangent space direction'' to make $C_{X}[n'-n] \to \Rbf \phi_* C_{X'}$ a map of graded complexes. The $g-k$ graded part of the complex $C_{X,\bullet}^\alpha$ looks like:
\begin{center}
\begin{tikzcd}
0 \arrow[r] & \Omega^{k}_X \ot S_0 \ot f^*\alpha \arrow[r] & \cdots \arrow[r] & \Omega^n_X \ot S_{n-k} \ot f^* \alpha \arrow[r] & 0
\end{tikzcd}
\end{center}
supported in degrees $[-(n-k), 0]$. Therefore, by the hypercohomology spectral sequence, $R^{k-n} p_{2*}$ is computed by taking the kernel of global sections in degree $-(n-k)$. Explicitly, 
\[ (R^{k-n} p_{2*} C_X^\alpha)_{g-k} = \ker{\left( H^0(X, \Omega^k_X \ot f^* \alpha) \ot S_0 \to H^0(X, \Omega^{k+1}_X \ot f^* \alpha) \ot S_1 \right)} \]
The same is true of $R^{k-n'} p_{2*} C_{X', g-k} = R^{k-n'} p_{2*} \R \phi_* C_{X', g-k}$. Applying $R^{k-n'}$ to $C_{X}[n'-n] \to \Rbf \phi_* C_{X'}$ (recalling the shift on the source) and taking $g-k$ graded piece we obtain
\begin{center}
    \begin{tikzcd}
    0 \arrow[r] & (R^{k-n} p_{2*} C_X^\alpha)_{g-k} \arrow[d, dashed] \arrow[r] & H^0(X, \Omega^k_X \ot f^* \alpha) \ot S_0 \arrow[r] \arrow[d] & H^0(X, \Omega^{k+1}_X \ot f^* \alpha) \ot S_1 \arrow[d]
    \\
    0 \arrow[r] & (R^{k-n} p_{2*} C_X^\alpha)_{g-k} \arrow[r] & H^0(X, \Omega^k_X \ot f^* \alpha) \ot S_0 \arrow[r] & H^0(X, \Omega^{k+1}_X \ot f^* \alpha) \ot S_1 
    \end{tikzcd}
\end{center}
By generic smoothness, the pullback of forms along $\phi$ is injective so the downward arrows are injections, and hence so is the dashed arrow, which is the map induced by the natural pullback. If $\phi$ is birational, the downward maps are isomorphism by the usual Hartog extension argument.
\end{proof}

Now we give the proof of Theorem~\ref{thm:generalization_of_PS14} which is a direct generalization of the main theorem of \cite{PS14} using the same method. Recall that we set
\[ P\Omega_X^k \coloneq \ker{(\Omega_X^k \ot S_0 \to \Omega_X^{k+1} \ot S_1)} = \cH^{k-n}(C_{X, g-k}) \]
and are given a line bundle subsheaf $\cN \embed P\Omega_X^k$ which is positive with respect to $f : X \to A$ meaning that $H^0(X, \cN^{\ot d} \ot f^* \L^{-1}) \neq 0$ for some $d > 0$ and ample $\L \in \Pic(A)$. Then we need to show that every global $1$-form pulled back along $f$ has a zero. 

\begin{proof}[Proof of Theorem~\ref{thm:generalization_of_PS14}]
We almost exactly repeat the proof in \cite{PS14} and when $k = n$ our argument is identical.
Recall that $C_X$ is supported on $Z_f \subset X \times V$ the locus of pairs $(x, \omega)$ so that $f^* \omega$ is zero at $x$. We exploit this in conjunction with the generic vanishing results on $\R f_* C_X$ surveyed above. Explicitly, \cite[Proposition 10.2]{PS14} shows that if we can construct a Hodge module $(\cM, F)$ and a graded $\Sym{\T_A}$-submodule $\F_\bullet \subset \gr^F_\bullet \cM$ such that
\begin{enumerate}
    \item there is a morphism $h : Y \to A$ from a smooth projective variety, such that $(\cM, F)$ is a direct summand of some $\cH^i h_+(\struct{Y}, F)$,
    \item the support of $\F$ on $T^* A$ is contained in $(f \times \id)(Z_f) \subset T^* A$, and
    \item for some $\ell \in \Z$, the sheaf $\F_{\ell}$ contains $\L \ot f_* \struct{X}$ as a subsheaf, where $\L$ is some ample line bundle on $A$,
\end{enumerate}
then $Z_f \to V$ is surjective proving the theorem.
\par 
First, as in \cite[Lemma 11.1]{PS14} we can find an isogeny $A \to A$ so that forming the fiber product gives an \etale cover $\tau : X' \to X$ such that $(\tau^* \cN \ot \tau^* f^* \L^{-1})^{\ot d}$ has a nonzero section for some $d \ge 0$ and some ample $\L \in \Pic(A)$. Of course, replacing $X$ by $X'$ it suffices to prove the result for $X'$, since $f^* \omega$ will be pulled back through the \etale cover $\tau$. 
\par 
Let $B = \cN \ot f^* \L^{-1}$ so by assumption there is a nonzero section $s \in H^0(X, B^{\ot d})$. We use this section to form a $d$-th cyclic cover $X_d \to X$ and resolve singularities $Y \to X_d$ to get a generically finite cover $\phi : Y \to X$ such that there exists a nonzero section $s' \in H^0(Y, \phi^* B)$ and that $(s')^{\ot d} = \phi^* s$. This gives a diagram
\begin{center}
    \begin{tikzcd}
        Y \arrow[rrd, bend right, "h"] \arrow[rr, bend left, "\phi"] \arrow[r] & X_d \arrow[r] & X \arrow[d, "f"]
        \\
        & & A
    \end{tikzcd}
\end{center}
producing a map
\[ \phi^* (C_X \ot B^{-1}) \to \phi^* C_X^\alpha \to C_Y \]
whose adjoint we denote by 
\[ \varphi : C_X \ot B^{-1} \to \Rbf \phi_* C_Y. \]
Now applying $R^{k-n} f_*$ for the projection $p : X \times V \to V$, we take the image to get
\begin{align*}
     \F_\bullet^{k-n} & \coloneq \im{(R^{k-n} f_* \varphi)} = \im{(R^{n-k} f_* (C_{X,\bullet} \ot B^{-1}) \to R^{n-k} h_* C_{Y,\bullet})} 
     \\ 
     & \cong \im{(\gr^F_\bullet \cH^{n-k} f_+(\struct{X}, F) \to \gr^F_\bullet \cH^{n-k} h_+(\struct{X}, F))}.
\end{align*}
Hence set $\cM = \cH^i h_+(\struct{Y}, F)$ which is a Hodge module whose graded part is $R^{n-k} h_* C_{Y,\bullet}$ by \cite[Proposition 2.11]{PS13} and projectivity of $h_+$. Then $\F^{k-n}_\bullet \subset \gr^F_\bullet \cM$ becomes a graded $\Sym{\T_A}$-submodule. Furthermore, the support of $R^{k-n} (f \times \id_V)_* (C_X \ot B^{-1})$ is contained in $(f \times \id_V)(Z_f)$ and hence the same is true of $\F^{k-n}$. Finally, we compute the $(g-k)$-graded part of $\F^{k-n}_\bullet$. Since  $\varphi : C_{X,g-k} \ot B^{-1} \to \Rbf \phi_* C_{Y,g-k}$ is a map of complexes supported in degrees $\ge (k-n)$, applying $R^{k-n} f_*$ we get
\[ \F^{k-n}_{g-k} \coloneq \im{(f_* (P\Omega_X^k \ot B^{-1}) \to h_* P \Omega_Y^k)} \]
since $f_*$ commutes with taking kernels. This map is injective since pushforward preserves injectivity. Furthermore, the inclusions 
\[ f^* \L \embed \cN \ot B^{-1} \embed P \Omega_X^k \ot B^{-1} \]
also remains injective when applying $f_*$ so we obtain an inclusion
\[ \L \ot f_* \struct{X} = f_* f^* \L \embed \F^{k-n}_{g-k} \] 
verifying the properties listed above needed to conclude the desired result.
\end{proof}

\subsection{Application to the MRC Fibration}

In this section we demonstrate how the results of the previous section imply imply trasfer of vanishing information about $1$-forms upwards along rational maps. Then we apply the story to the MRC fibration to conclude Theorem~\ref{thm:main_MRC}.

\begin{cor} \label{cor:PS14_rational_map}
Let $g : X' \rat X$ be a dominant rational map in $\Var_A$. Suppose that $X \to A$ factors birationally through the Iitaka fibration of $X$. Then $f : X' \to A$ does not satisfy $(\ast)_1$ i.e. every $\omega \in H^0(A, \Omega_A)$ satisfies $Z(f^* \omega) \neq \varnothing$.
\end{cor}

\begin{proof}
By hypothesis, the diagram
\begin{center}
    \begin{tikzcd}
     X' \arrow[r, "f'"] \arrow[d, dashed, "g"'] & A
     \\
     X \arrow[ru, "f"']
    \end{tikzcd}
\end{center}
commutes. There is a well-defined sheaf $g^* \omega_X$ because $g$ is defined in codimension $1$ and $\omega_X$ is a line bundle. Furthermore, pullback induces an inclusion $g^* \omega_X \embed P \Omega^k_{X'} \subset \Omega^k_{X'}$ where $k \coloneq \dim{X}$. This holds because the forms $f'^* \omega_i$ are pulled back through $g$ and therefore, for any local section $\theta$ of $\omega_X$ we have
\[ g^* \theta \wedge f'^* \omega_i = g^* (\theta \wedge f^* \omega_i) = 0 \]
since $\theta$ is a top form. Hence the natural pullback $g^* \omega_X \to \Omega_{X'}^k$, (which extends from the locus of definition since these are vector bundles) factors through the subsheaf $P \Omega_X^k$. Furthermore, the condition that $f : X \to A$ factors birationally through the Iitaka fibration is equivalent to $H^0(X, \omega_X^{\ot d} \ot f^* \L^{-1}) \neq 0$ for some $d \ge 0$. Since $g$ is dominant $H^0(X', (g^* \omega_X)^{\ot d} \ot f'^* \L^{-1}) \neq 0$, $f' : X' \to A$ satisfies the hypotheses of Theorem~\ref{thm:generalization_of_PS14} which concludes the argument. 
\end{proof}

Now we prove the main result of the section.

\begin{proof}[Proof of Theorem~\ref{thm:main_MRC}]
We want to apply Theorem~\ref{thm:Iitaka_decomposition} to the Iitaka fibration of a good minimal model $\psi : Y^{\min} \to S$ and $f : Y^{\min} \to A$. Therefore, it suffices to show that the general fiber of $\psi$ has image of dimension $\ge g$ in $A$. We repeat the argument proving Lemma~\ref{lemma:dimesnion_and_PLI_forms}. Let $W \subset H^0(A, \Omega_A)$ be a $g$-dimensional subspace of $1$-forms verifying $(\ast)_g$ for $f : X \to A$. Let $F$ be the generic fiber of $\psi$ and consider $Q = \coker{(\Alb_F \to \Alb_X)}$. Since $Y \to Q$ factors through the Iitaka fibration the diagram
\begin{center}
    \begin{tikzcd}
        X \arrow[r, "f"] \arrow[d, dashed] & Q
        \\
        Y \arrow[ru]
    \end{tikzcd}
\end{center}
satisfies the hypotheses of Corollary~\ref{cor:PS14_rational_map}. Therefore $X \to Q$ does not satisfy $(\ast)_1$.  Since $W$ consists of nowhere vanishing forms, $W \cap H^0(Q, \Omega_Q) = \{ 0 \}$ meaning 
\[ \dim{W} + \dim{Q} \le \dim{\Alb_X} \]
so the image of $\Alb_F$ has dimension at least $\dim{W} = g$. 
\end{proof}

\bibliographystyle{alpha}
\bibliography{refs}

\begin{bibdiv}
\begin{biblist}

\bib{Amb05}{article}{
      author={Ambro, Florin},
       title={The moduli {$b$}-divisor of an lc-trivial fibration},
        date={2005},
        ISSN={0010-437X,1570-5846},
     journal={Compos. Math.},
      volume={141},
      number={2},
       pages={385\ndash 403},
         url={https://doi.org/10.1112/S0010437X04001071},
      review={\MR{2134273}},
}

\bib{BCHM10}{article}{
      author={Birkar, Caucher},
      author={Cascini, Paolo},
      author={Hacon, Christopher~D.},
      author={McKernan, James},
       title={Existence of minimal models for varieties of log general type},
        date={2010},
        ISSN={0894-0347,1088-6834},
     journal={J. Amer. Math. Soc.},
      volume={23},
      number={2},
       pages={405\ndash 468},
         url={https://doi.org/10.1090/S0894-0347-09-00649-3},
      review={\MR{2601039}},
}

\bib{CCH23}{article}{
      author={{Chen}, Nathan},
      author={{Church}, Benjamin},
      author={{Hao}, Feng},
       title={{Nowhere vanishing holomorphic one-forms and fibrations over abelian varieties}},
        date={2023-06},
     journal={arXiv e-prints},
       pages={arXiv:2306.15064},
      eprint={2306.15064},
}

\bib{DJ74}{article}{
      author={Du~Bois, Philippe},
      author={Jarraud, Pierre},
       title={Une propri\'et\'e{} de commutation au changement de base des images directes sup\'erieures du faisceau structural},
        date={1974},
        ISSN={0302-8429},
     journal={C. R. Acad. Sci. Paris S\'er. A},
      volume={279},
       pages={745\ndash 747},
      review={\MR{376678}},
}

\bib{EGA}{article}{
      author={Dieudonn{\'e}, Jean},
      author={Grothendieck, Alexander},
       title={\'{E}l\'ements de g\'eom\'etrie alg\'ebrique},
        date={1961--1967},
     journal={Inst. Hautes \'Etudes Sci. Publ. Math.},
      volume={4, 8, 11, 17, 20, 24, 28, 32},
}

\bib{SGA3}{book}{
      author={Demazure, Michel},
      author={Grothendieck, Alexander},
       title={Sch\'emas en groupes. {I}: {P}ropri\'et\'es g\'en\'erales des sch\'emas en groupes},
      series={Lecture Notes in Mathematics},
   publisher={Springer-Verlag, Berlin-New York},
        date={1970},
      volume={Vol. 151},
        note={S\'eminaire de G\'eom\'etrie Alg\'ebrique du Bois Marie 1962/64 (SGA 3)},
      review={\MR{274458}},
}

\bib{Fuj10}{article}{
      author={Fujino, Osamu},
       title={Finite generation of the log canonical ring in dimension four},
        date={2010},
        ISSN={2156-2261,2154-3321},
     journal={Kyoto J. Math.},
      volume={50},
      number={4},
       pages={671\ndash 684},
         url={https://doi.org/10.1215/0023608X-2010-010},
      review={\MR{2740690}},
}

\bib{Hao23}{article}{
      author={Hao, Feng},
       title={Nowhere vanishing holomorphic one-forms on varieties of {K}odaira codimension one},
        date={2023},
     journal={To appear in Int. Math. Res. Not. IMRN},
}

\bib{HK05}{article}{
      author={Hacon, Christopher~D.},
      author={Kov\'{a}cs, S\'{a}ndor~J.},
       title={Holomorphic one-forms on varieties of general type},
        date={2005},
        ISSN={0012-9593},
     journal={Ann. Sci. \'{E}cole Norm. Sup. (4)},
      volume={38},
      number={4},
       pages={599\ndash 607},
         url={https://doi.org/10.1016/j.ansens.2004.12.002},
      review={\MR{2172952}},
}

\bib{HS21(1)}{article}{
      author={Hao, Feng},
      author={Schreieder, Stefan},
       title={Holomorphic one-forms without zeros on threefolds},
        date={2021},
        ISSN={1465-3060,1364-0380},
     journal={Geom. Topol.},
      volume={25},
      number={1},
       pages={409\ndash 444},
         url={https://doi.org/10.2140/gt.2021.25.409},
      review={\MR{4226233}},
}

\bib{Kawamata08}{article}{
      author={Kawamata, Yujiro},
       title={Flops connect minimal models},
        date={2008},
        ISSN={0034-5318,1663-4926},
     journal={Publ. Res. Inst. Math. Sci.},
      volume={44},
      number={2},
       pages={419\ndash 423},
         url={https://doi.org/10.2977/prims/1210167332},
      review={\MR{2426353}},
}

\bib{Kawamata81}{article}{
      author={Kawamata, Yujiro},
       title={Characterization of abelian varieties},
        date={1981},
        ISSN={0010-437X,1570-5846},
     journal={Compositio Math.},
      volume={43},
      number={2},
       pages={253\ndash 276},
         url={http://www.numdam.org/item?id=CM_1981__43_2_253_0},
      review={\MR{622451}},
}

\bib{KK10}{article}{
      author={Koll\'ar, J\'anos},
      author={Kov\'acs, S\'andor~J.},
       title={Log canonical singularities are {D}u {B}ois},
        date={2010},
        ISSN={0894-0347,1088-6834},
     journal={J. Amer. Math. Soc.},
      volume={23},
      number={3},
       pages={791\ndash 813},
         url={https://doi.org/10.1090/S0894-0347-10-00663-6},
      review={\MR{2629988}},
}

\bib{Kleiman}{incollection}{
      author={Kleiman, Steven~L.},
       title={The {P}icard scheme},
        date={2005},
   booktitle={Fundamental algebraic geometry},
      series={Math. Surveys Monogr.},
      volume={123},
   publisher={Amer. Math. Soc., Providence, RI},
       pages={235\ndash 321},
      review={\MR{2223410}},
}

\bib{KM98}{book}{
      author={Koll\'{a}r, J\'{a}nos},
      author={Mori, Shigefumi},
       title={Birational geometry of algebraic varieties},
      series={Cambridge Tracts in Mathematics},
   publisher={Cambridge University Press, Cambridge},
        date={1998},
      volume={134},
        ISBN={0-521-63277-3},
         url={https://doi.org/10.1017/CBO9780511662560},
        note={With the collaboration of C. H. Clemens and A. Corti, Translated from the 1998 Japanese original},
      review={\MR{1658959}},
}

\bib{Kollar09}{book}{
      author={Koll\'{a}r, J\'{a}nos},
       title={Lectures on resolution of singularities},
      series={Annals of Mathematics Studies},
   publisher={Princeton University Press, Princeton, NJ},
        date={2007},
      volume={166},
        ISBN={978-0-691-12923-5; 0-691-12923-1},
      review={\MR{2289519}},
}

\bib{KollarRatCurves}{book}{
      author={Koll\'ar, J\'anos},
       title={Rational curves on algebraic varieties},
      series={Ergebnisse der Mathematik und ihrer Grenzgebiete. 3. Folge. A Series of Modern Surveys in Mathematics [Results in Mathematics and Related Areas. 3rd Series. A Series of Modern Surveys in Mathematics]},
   publisher={Springer-Verlag, Berlin},
        date={1996},
      volume={32},
        ISBN={3-540-60168-6},
         url={https://doi.org/10.1007/978-3-662-03276-3},
      review={\MR{1440180}},
}

\bib{Lai11}{article}{
      author={Lai, Ching-Jui},
       title={Varieties fibered by good minimal models},
        date={2011},
        ISSN={0025-5831,1432-1807},
     journal={Math. Ann.},
      volume={350},
      number={3},
       pages={533\ndash 547},
         url={https://doi.org/10.1007/s00208-010-0574-7},
      review={\MR{2805635}},
}

\bib{LZ05}{incollection}{
      author={Luo, Tie},
      author={Zhang, Qi},
       title={Holomorphic forms on threefolds},
        date={2005},
   booktitle={Recent progress in arithmetic and algebraic geometry},
      series={Contemp. Math.},
      volume={386},
   publisher={Amer. Math. Soc., Providence, RI},
       pages={87\ndash 94},
         url={https://doi.org/10.1090/conm/386/07219},
      review={\MR{2182772}},
}

\bib{MP21}{article}{
      author={Meng, Fanjun},
      author={Popa, Mihnea},
       title={Kodaira dimension of fibrations over abelian varieties},
        date={2021},
     journal={arXiv preprint arXiv:2111.14165},
}

\bib{Nagata63}{article}{
      author={Nagata, Masayoshi},
       title={A generalization of the imbedding problem of an abstract variety in a complete variety},
        date={1963},
        ISSN={0023-608X},
     journal={J. Math. Kyoto Univ.},
      volume={3},
       pages={89\ndash 102},
         url={https://doi.org/10.1215/kjm/1250524859},
      review={\MR{158892}},
}

\bib{Prokhorov21}{article}{
      author={Prokhorov, Yuri~G.},
       title={Equivariant minimal model program},
        date={2021},
        ISSN={0042-1316,2305-2872},
     journal={Uspekhi Mat. Nauk},
      volume={76},
      number={3(459)},
       pages={93\ndash 182},
         url={https://doi.org/10.4213/rm9990},
      review={\MR{4265398}},
}

\bib{PS13}{article}{
      author={Popa, Mihnea},
      author={Schnell, Christian},
       title={Generic vanishing theory via mixed {H}odge modules},
        date={2013},
        ISSN={2050-5094},
     journal={Forum Math. Sigma},
      volume={1},
       pages={Paper No. e1, 60},
         url={https://doi.org/10.1017/fms.2013.1},
      review={\MR{3090229}},
}

\bib{PS14}{article}{
      author={Popa, Mihnea},
      author={Schnell, Christian},
       title={Kodaira dimension and zeros of holomorphic one-forms},
        date={2014},
        ISSN={0003-486X,1939-8980},
     journal={Ann. of Math. (2)},
      volume={179},
      number={3},
       pages={1109\ndash 1120},
         url={https://doi.org/10.4007/annals.2014.179.3.6},
      review={\MR{3171760}},
}

\bib{SY22}{article}{
      author={Schreieder, Stefan},
      author={Yang, Ruijie},
       title={Zeros of one-forms and holomorphically trivial fibrations},
        date={2022},
     journal={To appear in Mich. Math Journal},
}

\bib{Villadsen21}{article}{
      author={Villadsen, Mads~Bach},
       title={Kodaira dimension and zeros of holomorphic one-forms, revisited},
        date={2022},
        ISSN={1073-2780,1945-001X},
     journal={Math. Res. Lett.},
      volume={29},
      number={6},
       pages={1891\ndash 1898},
      review={\MR{4589382}},
}

\bib{Jinsong20}{article}{
      author={Xu, Jinsong},
       title={Homogeneous fibrations on log {C}alabi-{Y}au varieties},
        date={2020},
        ISSN={0025-2611,1432-1785},
     journal={Manuscripta Math.},
      volume={162},
      number={3-4},
       pages={389\ndash 401},
         url={https://doi.org/10.1007/s00229-019-01137-6},
      review={\MR{4109492}},
}

\end{biblist}
\end{bibdiv}

\end{document}